\documentclass[11pt]{amsart}
\usepackage{amsmath,amsthm}
\usepackage{amssymb}
\usepackage{pinlabel}
\usepackage[left = 3cm, right = 3 cm , top = 3 cm , bottom = 3cm, marginparwidth=2.5cm ]{geometry}
\usepackage[utf8]{inputenc}
\usepackage{cite}
\usepackage[textsize=small]{todonotes}
\usepackage{enumerate}
\usepackage{mathtools}
\usepackage{tikz}
\usepackage[hidelinks]{hyperref}

\newtheorem{thm}{Theorem}[section]
\newtheorem{prop}[thm]{Proposition}

\newtheorem{cor}[thm]{Corollary}

\newtheorem{question}[thm]{Question}
\newtheorem{example}[thm]{Example}

\newtheorem*{eulerclassone}{Euler Class One Conjecture}

\theoremstyle{remark}
\newtheorem{remark}[thm]{Remark}

\theoremstyle{definition}
\newtheorem{definition}[thm]{Definition}

\tikzset{every picture/.style=thick}
\usetikzlibrary{decorations.markings,intersections}

	\usepackage{imakeidx}
	\makeindex
\begin{document}
	\title{Thurston norm and the Euler class}
	\author[Mehdi Yazdi]{Mehdi Yazdi}
	\address{Department of Mathematics\\King's College London}
	\email{mehdi.yazdi@kcl.ac.uk}
	
	\date{}
	\thanks{}
	
	\begin{abstract}
		In his influential work, Thurston introduced a norm on the second homology group of compact orientable 3-manifolds $M$, which by duality also determines a dual norm on the second cohomology group. A natural question, initiated by Thurston, is whether integral points on the boundary of the dual norm ball have a geometric interpretation. Thurston showed that the Euler class of the oriented tangent plane field to any taut foliation of $M$ lies in the dual unit ball, and conjectured that, conversely, any integral point on the boundary of the dual unit ball is realised as the Euler class of a taut foliation. In this chapter, we discuss how several geometric, topological, and dynamical structures on a 3-manifold give rise to integral points in the dual unit ball of the Thurston norm, and what is known about Thurston's Euler class one conjecture in these contexts. These structures are taut foliations, tight contact structures, pseudo-Anosov flows, quasigeodesic flows, and circular orders on the fundamental group.
	\end{abstract}	
	
	\maketitle

	
	\section{Introduction}
	
	Let $M$ be a compact orientable 3-manifold. Thurston\index{Thurston, William} \cite{thurston1986norm} defined a semi-norm on the second homology group $H_2(M ; \mathbb{R})$ (respectively $H_2(M , \partial M ; \mathbb{R})$) as follows. Given a connected compact orientable surface $S$, define the \emph{negative part of the Euler characteristic} as 
	\[ \chi_-(S) := \max \{ 0 , - \chi(S) \}. \]
	If $S$ is disconnected, with connected components $S_1, \cdots, S_k$, then define 
	\[ \chi_-(S) := \sum_{i=1}^{k} \chi_-(S_i). \]
	In words, $\chi_-(S)$ is obtained by taking the absolute value of the Euler characteristic after deleting any component of $S$ that has positive Euler characteristic (namely any sphere and disc components). Given an integral point $a \in H_2(M ; \mathbb{R})$ (respectively $H_2(M , \partial M ; \mathbb{R})$), define the norm $x(a)$ of $a$ as 
	\[ x(a) := \min \{ \chi_-(S) \hspace{1mm} | \hspace{1mm} [S]= a, \text{ and }S \text{ is a compact properly embedded oriented surface in } M\}. \]
	The norm of a rational point $a$ is defined by scaling, and then the norm is extended to real points continuously. The Thurston norm\index{Thurston!norm} is generally a semi-norm, and the subspace of $H_2(M ; \mathbb{R})$ or $H_2(M , \partial M ; \mathbb{R})$ with norm $0$ is spanned by homology classes of essential surfaces with non-negative Euler characteristics. It follows that if $M$ contains no essential spheres, discs, tori, and annuli, then $x$ is a norm. 
	
	More generally, given a subsurface $N \subset \partial M$, the Thurston norm can be defined on $H_2(M, N; \mathbb{R})$, by considering surfaces whose boundary lie in $N$. Given a compact properly embedded oriented surface $(S, \partial S) \subset (M, N)$ we say $S$ is \emph{norm-minimising in $H_2(M, N)$}\index{norm-minimising}\index{surface!norm-minimising} if $S$ is incompressible and $\chi_-(S) = x([S, \partial S])$.
	
	The Thurston norm on $H_2(M ; \mathbb{R})$ naturally defines a dual norm\index{Thurston!dual norm} on the dual vector space $H^2(M ; \mathbb{R})$ by the formula 
	\[ x^*(u) : = \sup_{0 \neq a \in H_2(M)} \frac{ \langle u , a \rangle}{x(a)},  \]
	where the pairing $\langle \ , \ \rangle $ is between the second cohomology and homology groups. The dual norm on $H^2(M , \partial M ; \mathbb{R})$ is similarly defined. Note that if $x$ is a semi-norm but not a norm, the dual norm could take the values $\pm \infty$ as well. Denote the unit ball of the Thurston norm by $B_{x}$ and the unit ball of the dual norm by $B_{x^*}$. 
	
	Let $d$ be the dimension of the vector space $H_2(M ; \mathbb{R})$. There is a natural embedding of the lattice $\mathbb{Z}^d$ in $H_2(M ; \mathbb{R})$ corresponding to the change of coefficients map $H_2(M ; \mathbb{Z}) \rightarrow H_2(M ; \mathbb{R})$. By definition, every point in the lattice $\mathbb{Z}^d \subset \mathbb{R}^d$ has integer norm. Thurston \cite{thurston1986norm} showed that any norm on $\mathbb{R}^d$ that takes integer values on the lattice $\mathbb{Z}^d$ has unit ball a (possibly non-compact) convex polyhedron and the unit ball of its dual norm is a compact convex polytope with integral vertices. In particular, the unit ball in $H^2(M ; \mathbb{R})$ (respectively in $H^2(M , \partial M ; \mathbb{R})$) of the dual Thurston norm is a compact convex polytope with integral vertices.
	
	Assume that $M$ is irreducible and closed. There are generalisations for the case that the boundary of $M$ is a union of tori. Let $\mathcal{F}$ be a taut foliation of $M$ and $S$ be an embedded incompressible surface in $M$. Roussarie \cite{roussarie1974plongements} and Thurston \cite{thurston1972foliations,thurston1986norm} showed that $S$ can be isotoped such that the induced singular foliation on $M$ has only finitely many singularities all of which are of saddle type. Denote the Euler class\index{Euler class!of a foliation} of the oriented tangent plane field to the foliation $\mathcal{F}$ by $e(\mathcal{F}) \in H^2(M ; \mathbb{R})$. Thurston derived an index sum formula\index{Euler class!index sum formula}\index{index sum formula} for the evaluation of $e(\mathcal{F})$ on the homology class $[S]$ of $S$ and used it to deduce that the following inequality\index{Thurston!inequality} holds
	\begin{eqnarray}
		|\langle e(\mathcal{F}), [S] \rangle | \leq |\chi(S)|.
		\label{eq:Thurston inequality}
	\end{eqnarray}  
	This translates to the statement that the Euler class $e(\mathcal{F})$ has dual Thurston norm at most one. Moreover, Thurston observed that if $\mathcal{F}$ has any compact leaf of negative Euler characteristic, then the equality holds in (\ref{eq:Thurston inequality}). Conversely, he conjectured \cite[p. 129, Conjecture 3]{thurston1986norm} that any integral element $a \in H^2(M ; \mathbb{R})$ of dual norm one is the Euler class of a taut foliation on $M$. 
	
	If $\xi$ is a transversely oriented plane field on a closed orientable 3-manifold, then the integral Euler class $e(\xi)$ of $\xi$ lies in $2 H^2(M ; \mathbb{Z})$. As a corollary, the real Euler class $e(\xi)$ also lies in the image of the map $2 H^2(M ; \mathbb{Z}) \rightarrow H^2(M ; \mathbb{R})$. We call this the \emph{parity condition}\index{Euler class!parity condition for the Euler class of a transversely oriented foliation on an orientable 3-manifolds}\index{parity condition}, which goes back at least to Wood \cite{wood1969foliations}. Thurston was aware of the parity condition and so we assume this condition as part of the hypotheses of his conjecture.

\begin{eulerclassone}[Thurston - 1976]
	Let $M$ be a closed orientable irreducible atoroidal 3-manifold with positive first Betti number. For any integral class $a \in H^2(M ; \mathbb{R})$ of dual Thurston norm one and satisfying the parity condition, there exists a taut foliation $\mathcal{F}$ of $M$ with Euler class $e(\mathcal{F})$ equal to $a$. 
\end{eulerclassone}
	
	A 3-manifold is called \emph{Haken}\index{Haken} if it is irreducible and if it contains a two-sided incompressible surface. A 3-manifold is \emph{atoroidal}\index{atoroidal} if every embedded incompressible torus in it is boundary-parallel. A 3-manifold is \emph{hyperbolic}\index{hyperbolic!manifold} if it admits a complete hyperbolic metric of finite volume, i.e. a complete Riemannian metric of constant sectional curvature $-1$ and with finite volume. Hyperbolic 3-manifolds are irreducible and atoroidal. By Thurston's Hyperbolisation Theorem\index{Thurston!hyperbolisation theorem} for Haken manifolds, any closed atoroidal Haken 3-manifold is hyperbolic. Therefore, in the above conjecture one can assume that $M$ is hyperbolic. 
	
	Novikov \cite{novikov1965topology} showed that every leaf of a taut foliation (compact or not) is incompressible. Thurston \cite{thurston1986norm} used inequality (\ref{eq:Thurston inequality}) to show that if $S$ is a compact leaf of a taut foliation on $M$, then $S$ is norm-minimising\index{norm-minimising!compact leaves of taut foliations are norm-minimising}. Gabai \cite{gabai1983foliations} proved a converse to this statement. Namely, let $M$ be a compact orientable irreducible 3-manifold with boundary a (possibly empty) union of tori, and $S$ be a norm-minimising surface in $M$. Assume that $\partial S$ is \emph{coherently oriented}\index{coherently oriented}, meaning that for every torus component $T \subset \partial M$, $\partial S$ (equipped with the boundary orientation induced from $S$) intersects $T$ in (a possibly empty) collection of curves that have consistent orientations. Then there is a taut foliation $\mathcal{F}$ of $M$ that has $S$ as a union of compact leaves, and such that $\mathcal{F}$ intersects $\partial M$ transversely and the restriction of $\mathcal{F}$ to each component of $\partial M$ has no two-dimensional Reeb component. Note that the condition of $\partial S$ being coherently oriented is necessary for such a foliation $\mathcal{F}$ to exist, because of the transverse orientability of $\mathcal{F}$. Gabai used this theorem to show that Thurston's Euler class one conjecture holds for vertices of the dual unit ball. See \cite[Remark 7.3]{gabai1997problems}, or \cite{gabai2020fullymarked} for a proof. 
	
	\begin{thm}[Euler Class One for Vertices]\index{Euler class!Euler class one theorem for vertices}
		Let $M$ be a compact orientable irreducible 3-manifold with boundary a (possibly empty) union of tori, and with positive first Betti number. Every vertex of the unit ball of the dual Thurston norm is realised as the Euler class of some taut foliation on $M$. 
		\label{thm:Gabai-euler class one for vertices}
	\end{thm}
	
	In fact, let $w$ be a vertex of the dual norm ball $B_{x^*}$ and $\mathcal{C}$ be the face of the Thurston norm ball $B_{x}$ dual to $v$. Since $v$ is a vertex, $\mathcal{C}$ is top-dimensional. Pick a norm-minimising surface $S$ whose homology class $[S]$ lies in the cone over the interior of the face $\mathcal{C}$ and such that $\partial S$ is coherently oriented. By Gabai's theorem, there is a taut foliation $\mathcal{F}$ that has $S$ as a compact leaf. It turns out that regardless of how the foliation $\mathcal{F}$ outside of the leaf $S$ looks like, the Euler class $e(\mathcal{F})$ is equal to $w$; here the assumption about $\mathcal{C}$ being top-dimensional is used.

	In \cite{yazdi2020thurston}, the author constructed the first counterexamples to the Euler class one conjecture assuming the Fully Marked Surface Theorem. The constructed manifolds have first Betti number $2$, which is the smallest possible since by Theorem \ref{thm:Gabai-euler class one for vertices} the conjecture holds for 3-manifolds with first Betti number equal to one. Moreover the unit ball of the Thurston norm for the constructed counterexamples has a simple diamond shape and the unit dual ball is a rectangle in suitable integral coordinates. 
	
	Let $\mathcal{F}$ be a taut foliation on an orientable 3-manifold $M$.  A compact properly embedded orientable incompressible surface $S$ in $M$ is called \emph{algebraically fully marked with respect to $\mathcal{F}$}\index{fully marked surface!algebraically}\index{surface!algebraically fully marked} if the equality in (1) happens. Thurston's proof of Inequality (\ref{eq:Thurston inequality}) indeed shows that any algebraically fully marked surface is norm-minimising. Thurston observed that compact leaves of taut foliations are fully marked, and similarly a union of compact leaves is fully marked if the members of the union are oriented coherently, i.e. if the transverse orientation of the surface always agrees with the transverse orientation of the foliation or that it always disagrees. The converse is not true. To see this note that every taut foliation of a hyperbolic 3-manifold can be perturbed so that the new foliation has no compact leaves. Since the Euler class is invariant under a homotopy of the plane field of the foliation, the new foliation has the same Euler class as the original one. Hence any fully marked surface with respect to the initial foliation remains fully marked with respect to a new taut foliation with no compact leaf. It is natural to instead ask if a converse holds up to homotopy of plane fields of taut foliations. This is the content of the Fully Marked Surface Theorem, under some additional hypothesis. In \cite{gabai2020fullymarked} Gabai and the author proved the Fully Marked Surface Theorem, thereby giving a negative answer to Thurston's Euler class one conjecture. 
	 
	 \begin{thm}[Fully Marked Surface Theorem]
	 	Let $M$ be a closed hyperbolic 3-manifold, $\mathcal{F}$ be a taut foliation of $M$, and $S$ be an algebraically fully marked surface in $M$. There is a taut foliation $\mathcal{G}$ of $M$ and an embedded surface $S'$ in $M$ such that 
	 	\begin{enumerate}
	 		\item $S'$ is homologous to $S$; 
	 		\item the oriented tangent plane fields of $\mathcal{F}$ and $\mathcal{G}$ are homotopic through plane fields; and 
	 		\item $S'$ is a  union of compact leaves of $\mathcal{G}$.
	 	\end{enumerate} 
	 	\label{thm:fully marked surface}
	 \end{thm}
 
	In particular if $S$ is the unique norm-minimising surface in its homology class, then we can take $S' = S$ up to isotopy. 
 	
 	Knowing that Thurston's Euler class one conjecture has a negative answer for taut foliations, a natural question is which cohomology classes are realised as Euler classes of taut foliations. At the time of this writing, we do not have a conjectural answer to this question. 
 	
 	In addition to taut foliations, there are several other topological, geometric, and dynamical structures on a 3-manifold that have an associated Euler class, and such that the dual Thurston norm of their Euler class is at most one. In particular, an analogue of Thurston's inequality (\ref{eq:Thurston inequality}) is known for 
 	\begin{enumerate}
 		\item[a)] tight contact structures, 
 		\item[b)] pseudo-Anosov flows on atoroidal 3-manifolds (respectively quasigeodesic flows on hyperbolic 3-manifolds), and
 		\item[c)] group actions on the circle.
 	\end{enumerate}
	
	Thurston's Euler class one conjecture can be understood as asking for a realisation of integral points on the boundary of the unit ball of dual Thurston norm by interesting topological, geometric, or dynamical structures. In this chapter, we discuss what is known about the analogue of Thurston's conjecture in other contexts, and discuss the status of questions raised in \cite{yazdi2020thurston}. We will see that in the case of tight contact structures, and also for group actions on the circle, it is known that the Euler class one conjecture is weaker than the original Euler class one conjecture for taut foliations. 
	
	In \cite{yazdi2020thurston}, the author asked if the Euler class one conjecture holds for tight contact structures. With Steven Sivek \cite{sivek2023thurston}, we showed that the constructed counterexample cohomology classes in \cite{yazdi2020thurston} are realised as Euler classes of (possibly negative) tight and weakly symplectically fillable contact structures. Recently Yi Liu \cite{liu2024euler} proved the following result. 
	
	\begin{thm}[Liu]
		For every oriented closed hyperbolic 3–manifold $M$, there exists some connected finite cover $\tilde{M}$ of $M$, and some even lattice point $\tilde{w} \in H^2(\tilde{M}; \mathbb{R})$ of dual Thurston norm one such that $\tilde{w}$ is not the real Euler class of any weakly symplectically fillable contact structure on $\tilde{M}$.
		\label{thm:Liu-contact}
	\end{thm}
	
	In particular, his result produces many new counterexamples to the Euler class one conjecture for taut foliations.  
	
	\begin{cor}[Liu]
		Every closed hyperbolic 3-manifold has a finite cover for which the Euler class one conjecture (for taut foliations) does not hold. 
		\label{cor:Liu-taut-foliation}
	\end{cor} 

	In \cite{yazdi2020thurston}, the author asked if a virtual version of the Euler class one conjecture holds. See Question \ref{que:virtual-taut foliation}. Liu \cite{liu2024criterion} gave a criterion in terms of Alexander polynomials that, when satisfied, it gives a positive answer to the virtual Euler class one conjecture. 
	
	\begin{thm}[Liu]
		Let $M$ be a closed oriented hyperbolic 3-manifold. Denote by $B_x$ the unit ball of the Thurston norm of $M$, and by $B_{x^*}$ the unit ball of the dual norm. Let $F \subset \partial B_{x^*}$ and $F^\wedge \subset \partial B_{x}$ be a dual pair of closed faces. Suppose that $\psi \in H^1(M ; \mathbb{Z}) $ is a primitive cohomology class such that the Poincar\'{e} dual of $\psi$ lies in the cone over the interior of $F^\wedge$. 
		If the Alexander polynomial\index{Alexander polynomial} $\Delta_M^{\psi}(t)$ does not vanish, then for any rational point $w$ in $F$, there exists some finite cyclic cover $\tilde{M}$ of $M$ dual to $\psi$, such that the pullback of $w$ to $\tilde{M}$ is the real Euler class of some taut foliation on $\tilde{M}$.
		\label{thm: Liu-virtual-taut foliation}
	\end{thm}
	
	As a corollary, Liu gives examples of hyperbolic 3-manifolds with first Betti numbers 2 and 3 respectively such that every \emph{rational} point on the boundary of the unit ball of the dual Thurston norm is virtually realised as the real Euler class of a taut foliation.
	
	\subsection{Outline} 
	
	In Section \ref{sec: taut foliations} we review the background on taut foliations and their Euler classes. In Section \ref{sec: fully marked} we discuss the fully marked surface theorem briefly, and show with an example why in general replacing the foliation preserving the homotopy class of its tangent plane field is necessary. Section \ref{sec: contact structures} discusses contact structures, and the analogue of Inequality (\ref{eq:Thurston inequality}) in this context due to Eliashberg. We also briefly mention some of the tools developed by Liu in his proof of Theorem \ref{thm:Liu-contact}.  In Section \ref{sec: flows}, pseudo-Anosov flows and quasigeodesics flows are reviewed. In Section \ref{sec: actions on cricle}, we talk about group actions on the circle, and the Milnor--Wood inequality, as well as its reformulation in terms of bounded cohomology. Section \ref{sec: virtual} discusses a virtual version of the Euler class one conjecture, and gives a sketch of Liu's proof of Theorem \ref{thm: Liu-virtual-taut foliation}. 
	
	An important topic that is not discussed here is the \emph{adjunction inequality}. Perhaps a satisfactory theory that unifies Thurston's inequality (\ref{eq:Thurston inequality}) and Eliashberg's inequality (Theorem \ref{thm:eliashberg-inequality}) would be through some form of the adjunction inequality, but such a theory is not developed yet. The book \cite{ozbagci-stipsicz} is an excellent reference for this topic. 
	
	\section{Acknowledgment} I would like to thank Alessandro Cigna for helpful comments on this chapter. 
	
	\section{Taut foliations}
	\label{sec: taut foliations}
	
	\subsection{Foliations}
	A codimension-one \emph{foliation}\index{foliation} of a closed 3-manifold $M$ is a decomposition of $M$ into injectively immersed surfaces such that locally it has the product form $\mathbb{R}^2 \times \mathbb{R}$ by surfaces $\mathbb{R}^2 \times \{ \text{point}\}$. The connected components of the surfaces in the decomposition are called the \emph{leaves}\index{foliation!leaf of} of the foliation. When $M$ has non-empty boundary, we often require that for each boundary component $T$ of $M$ the foliation is either transverse to $T$ or has $T$ as a leaf. 
	
	\begin{thm}[Reeb stability theorem \cite{reeb1952certaines}]\index{Reeb stability}
		Let $\mathcal{F}$ be a transversely orientable codimension-one foliation of a compact manifold. Assume that $\mathcal{F}$ has a leaf $L$ with $\pi_1(L)$ finite. Then either $M$ is $L \times [0,1]$ and $\mathcal{F}$ is the product foliation, or $M$ fibers over $S^1$ with fiber $L$ and $\mathcal{F}$ is the fibration. 
	\end{thm}

	For example, if $\mathcal{F}$ is a transversely orientable foliation of a compact orientable 3-manifold that has a leaf diffeomorphic to the 2-sphere $S^2$ (respectively 2-disc $D^2$) then $M$ is diffeomorphic to $S^2 \times S^1$ (respectively $D^2 \times S^1$, assuming further that $\mathcal{F}$ is transverse to the boundary) with the product foliation.
	
	\subsection{Suspension construction}
	
	\index{suspension!construction}Let $B$ and $F$ be manifolds, where $B$ is connected, and $b_0$ be a base point in $B$. Let $\rho \colon \pi_1(B,b_0) \rightarrow \mathrm{Homeo}(F)$ be a homomorphism. Then $\rho$ defines a bundle with total space $E$, base $B$, and fiber $F$, and a foliation $\mathcal{F}$ of $E$ transverse to the fibers as follows. Let $\tilde{B}$ be the universal cover of $B$, and consider $\tilde{B} \times F$ with the product foliation whose leaves are $\tilde{B} \times \{ \text{point} \}$. Let
	\[ E : = (\tilde{B} \times F)/ \pi_1(B , b_0). \]
	Here the action of $\pi_1(B , b_0)$ on $\tilde{B} \times F$ is defined as 
	\[ \gamma \cdot (\tilde{b} , f) : = (\gamma \cdot \tilde{b}, \rho(\gamma)(f) ), \]
	where the action on the first factor is by covering transformations. The fibration of $\tilde{B} \times F$ by fibers $F$ induces a fibration on $E$ since the action of $\pi_1(B , b_0)$ preserves the set of fibers. Moreover, the product foliation on $\tilde{B} \times F$ descends to a foliation $\mathcal{F}$ on $E$ transverse to the fibers, since the action preserves the leaves of the product foliation on $\tilde{B} \times F$. Each leaf of the foliation $\mathcal{F}$ is a covering of $B$, since it is a quotient of $\tilde{B}$. The foliation $\mathcal{F} = \mathcal{F}_\rho$ constructed above is called the \emph{suspension foliation associated with $\rho$}\index{foliation!suspension foliation}\index{suspension!foliation}.
	
	\subsection{Taut foliations}
	A transversely orientable codimension-one foliation is called \emph{taut}\index{foliation!taut}\index{taut!foliation} if for every point $p \in M$ there is a closed loop $\gamma_p \colon S^1 \rightarrow M$ that passes through $p$ and is transverse to the foliation at each point. Being \emph{transverse} means that every point $q$ in the image of $\gamma_p$ has a neighbourhood $U $ homeomorphic to $(-1, 1)^2 \times (-1, 1)$ foliated by copies of $(-1, 1)^2 \times \{ \text{point} \}$ such that the image of $\gamma_p$ in $U$ is the arc $\{ (0, 0) \} \times (-1, 1)$.

	\subsection{Relative Euler class}
	Let $\mathcal{F}$ be a transversely oriented foliation of a compact oriented 3-manifold such that each component $T$ of $\partial M$ is either a leaf of $\mathcal{F}$ or is transverse to $\mathcal{F}$. Assume that $\partial M$ is a union of tori; this is automatic for any component of $\partial M$ that is transverse to $\mathcal{F}$. In this case there is a well-defined relative Euler class\index{relative Euler class!of a foliation} $e(\mathcal{F}) \in H^2(M , \partial M)$ for the oriented tangent plane bundle to the foliation. To define a relative Euler class we need to define a section (or trivialisation) of $T \mathcal{F}$ on $\partial M$, where we allow $\partial M$ to be empty. Fix a Riemannian metric on $M$. First assume that $\mathcal{F}$ is transverse to $\partial M$: in this case consider the section defined on $\partial M$ that lies inside $T \mathcal{F} \cap T(\partial M)$, has unit length, and whose orientation is the boundary orientation induced from leaves of $\mathcal{F}$. This defines the relative Euler class when $\mathcal{F}$ is transverse to $\partial M$. When some components of $\partial M$ are leaves of $\mathcal{F}$, there is no canonical section of $T \mathcal{F}$ on $\partial M$. However, for each torus component $T$ of $\partial M$ that is a leaf of $\mathcal{F}$, we can consider an identification $T \cong S^1 \times S^1$ which gives a trivialisation of the tangent bundle of $T$ as $TS^1 \times TS^1$. It can be shown that different identifications $T \cong S^1 \times S^1$ give rise to the same trivialisation of the tangent bundle of $T$ up to homotopy, in turn allowing us to define a relative Euler class.  See \cite[Section 3]{yazdi2020thurston} for the details. 

	\subsection{Roussarie--Thurston general position}\index{Roussarie--Thurston general position}
	Let $\mathcal{F}$ be a transversely oriented foliation of a compact oriented 3-manifold such that each component $T$ of $\partial M$ is either a leaf of $\mathcal{F}$ or is transverse to $\mathcal{F}$. Let $S$ be a connected compact properly embedded orientable surface in $M$. Assume that each component of $\partial S$ is either transverse to $\mathcal{F}$ or lies in a leaf of $\mathcal{F}$. Then a general position argument shows that $S$ can be isotoped such that it is transverse to $\mathcal{F}$ except at finitely many points of tangencies that are either centers or saddles. Roussarie \cite{roussarie1974plongements} and Thurston \cite{thurston1986norm} showed that if $\mathcal{F}$ is Reebless and $S$ is incompressible and boundary-incompressible then $S$ can be isotoped such that it is transverse to the foliation except at finitely many points of saddle and circle tangencies. Thurston showed \cite{thurston1986norm} that if $\mathcal{F}$ is taut and $S$ is as before, then $S$ can be isotoped such that it is either a leaf of $\mathcal{F}$ or it is transverse to $\mathcal{F}$ except at finitely many points of saddle tangencies. See Candel and Conlon \cite{candel2003foliations} for a proof of Roussarie general position, or \cite{gabai2000combinatorial} for a proof of a generalisation of Thurston's general position for immersed incompressible surfaces in $C^0$ taut foliations. 
	
	\subsection{Thurston's inequality}
	
	Let $\mathcal{F}$ be a transversely oriented foliation of a compact oriented 3-manifold such that each component $T$ of $\partial M$ is either a leaf of $\mathcal{F}$ or is transverse to $\mathcal{F}$. Let $S$ be a connected compact properly embedded oriented surface in $M$. Then $e(\mathcal{F})$ associates a number to the homology class $[S] \in H_2(M , \partial M)$ of $S$, which we denote by $\langle e(\mathcal{F}) , [S] \rangle$. Thurston observed that this number is equal to $\chi(S)$ if $S$ is a leaf of $M$ whose positive normal vector agrees with the transverse orientation of $\mathcal{F}$. Moreover, he proved an index sum formula for the value of $\langle e(\mathcal{F}) , [S] \rangle$ when $S$ is transverse to $\mathcal{F}$ except at finitely many points of center, saddle, or circle tangencies. Given an isolated point $s$ of tangency between $S$ and $\mathcal{F}$, define the index $i(s)$ of $s$ as $+1$ if $s$ is a center tangency, and $-1$ if $s$ is a saddle tangency. We say that $s$ is of \emph{positive type} if the transverse orientations of $S$ and $\mathcal{F}$ at $s$ agree with each other, and of \emph{negative type} if they disagree. Denote 
	\[ I_P :=  \sum_{s \text{ of positive type}} i(s), \]
	and 
	\[ I_N := \sum_{s \text{ of negative type}} i(s). \]
	
	Thurston showed that the following index sum formula holds: 
	\begin{align}
		\langle e(\mathcal{F}) , [S] \rangle = I_P - I_N.
	\end{align}
	In particular, circle tangencies do not contribute to the value of $\langle e(\mathcal{F}) , [S] \rangle$. This index sum formula is obtained by defining an explicit section of $T\mathcal{F} |S$ on the complement of the points of tangency between $\mathcal{F}$ and $S$, and then calculating the local obstructions for extending the section over the singular points. The value $\langle e(\mathcal{F}) , [S] \rangle$ is then equal to the sum of the values of local obstruction. 
	
	Moreover, by the Poincar\'{e}--Hopf formula\index{Poincar\'{e}--Hopf formula} for the induced singular foliation $\mathcal{F}|S$ on $S$ we have 
	\begin{align}
		\chi(S) = I_P + I_N.
		\label{eq: Poincare-Hopf}
	\end{align}

	Now assume that $\mathcal{F}$ is taut and $S$ is incompressible and boundary-incompressible with $\chi(S) \leq 0$, and that each component of $\partial S$ is either transverse to $\mathcal{F}$ or lies in a leaf of $\mathcal{F}$. By Roussarie--Thurston general position, we can isotope $S$ such that each component of $S$ is either a leaf or it is transverse to the foliation except at finitely many saddle points of tangency. If $S$ is a leaf of $\mathcal{F}$ and the transverse orientations of $S$ and $\mathcal{F}$ agree, then $\langle e(\mathcal{F}) , [S] \rangle = \chi(S)$, since the restriction of $T \mathcal{F}$ to $S$ gives the tangent bundle of $S$, and by Hopf's theorem we have $\langle e(TS), [S] \rangle = \chi(S)$. So assume that $S$ is transverse to the foliation except at finitely many saddle points of tangency. In this case, $i(s)=-1$ for each tangency point, and so we have 
	\[ |\langle e(\mathcal{F}) , [S] \rangle | = |I_P- I_N| \leq |I_P + I_N| = |\chi(S)|. \]
	In the above we used the fact that each term in the sum $I_P + I_N$ is $-1$, and so switching some of the terms to $+1$ instead decreases the sum in absolute value. This completes Thurston's proof of Inequality (\ref{eq:Thurston inequality}). Note that by the Poincar\'{e}--Hopf formula (\ref{eq: Poincare-Hopf}) if every point of tangency between $S$ and $\mathcal{F}$ is either a center or a saddle, then there are at least $|\chi(S)|$ tangencies, with equality if and only if all points of tangencies are saddles. By Roussarie--Thurston general position, when $S$ is incompressible and boundary-incompressible and $\mathcal{F}$ is taut, the surface $S$ can be isotoped such that the number of tangencies is exactly this minimum number $|\chi(S)|$. 
	
	\subsection{Finite depth foliations}
	
	A leaf $L$ of a foliation $\mathcal{F}$ of a manifold $M$ is of \emph{depth 0}\index{leaf!depth of}\index{depth!of a leaf} if it is compact. A leaf $L$ is of depth $(k+1)$ if it is not of depth at most $k$ and the limit points of $L$ (as a subspace of $M$) are a union of leaves of depth at most $k$. The \emph{depth of a foliation}\index{foliation!depth of}\index{depth!of a foliation} is the smallest number $k \in \mathbb{N}$ such that every leaf of $\mathcal{F}$ is of depth at most $k$, and it is equal to $+ \infty$ if no such number exists. In the former case we say that the foliation is of \emph{finite depth}\index{depth!finite}\index{foliation!finite depth}. 
	For example, in the Reeb foliation of the solid torus, the boundary torus is of depth $0$ and every other leaf is of depth $1$.  

	\begin{remark}
		Note that if $\mathcal{F}$ is a finite depth taut foliation of a finite-volume hyperbolic 3-manifold $M$ then the relative Euler class $e(\mathcal{F})$ of $M$ has dual Thurston norm exactly one. This is because by Thurston's inequality (\ref{eq:Thurston inequality}) the dual norm of $e(\mathcal{F})$ is at most one. Moreover, every compact leaf $S$ of $\mathcal{F}$ satisfies $ \langle e(\mathcal{F}) , [S] \rangle  = \chi(S)$. By Novikov's theorem $S$ is incompressible, and since $M$ is hyperbolic we must have $\chi(S)<0$. Therefore, the dual norm of $e(\mathcal{F})$ is equal to one. 
		
		This shows that some familiar foliations are not of finite depth. For example, if $\mathcal{F}$ is the weak stable (or unstable) foliation\index{foliation!weak (un)stable} of an Anosov flow\index{flow!Anosov} on a closed orientable hyperbolic 3-manifold, then the flow direction defines a non-zero section of $T \mathcal{F}$, and so $e(\mathcal{F}) = 0 \in H^2(M ; \mathbb{Z})$ whenever $\mathcal{F}$ is transversely orientable. Therefore, $\mathcal{F}$ is not of finite depth. Note that $\mathcal{F}$ is taut since every leaf of $\mathcal{F}$ is homeomorphic to either an open annulus or a plane, in particular all leaves are non-compact. 
	\end{remark}

	David Gabai \cite{gabai1983foliations} constructed finite depth taut foliations on a large class of 3-manifolds. 

	\begin{thm}[Gabai]
		Let $M$ be a compact orientable irreducible boundary-irreducible 3-manifold with $H_2(M , \partial M ; \mathbb{R}) \neq 0$. Let $S$ be a compact orientable surface properly embedded in $M$ such that $S$ is incompressible and Thurston norm-minimising in its homology class and $\partial S$ is coherently oriented. There is a finite depth taut foliation $\mathcal{F}$ on $M$ such that $S$ is a union of compact leaves of $\mathcal{F}$.
		\label{thm:norm-minimising surfaces are leaves}
	\end{thm}
	
	In order to construct taut foliations, Gabai \cite{gabai1983foliations} introduced the notion of sutured manifolds. These are manifolds with extra data on their boundary that restrict the way a foliation is allowed to intersect the boundary. 
	
	\begin{definition}[Sutured manifold]
		A \emph{sutured manifold}\index{sutured manifold} $(M, \gamma)$ is a compact oriented 3-manifold $M$ together with a set of $\gamma \subset \partial M$ of pairwise disjoint annuli $A(\gamma)$ and tori $T(\gamma)$. Every component of $\partial M - \gamma$ is oriented. We denote by $R_+(\gamma)$ (respectively $R_-(\gamma)$) the union of components of $\partial M - \gamma$ whose normal vector points out of (respectively into) $M$. Each annulus component of $\gamma$ must be adjacent to both $R_+(\gamma)$ and $R_-(\gamma)$. 
		
		A sutured manifold is \emph{taut}\index{sutured manifold!taut}\index{taut!sutured manifold} if $M$ is irreducible and $R(\gamma)$ is norm-minimising in $H_2(M , \gamma)$. 
	\end{definition}

For example if $X$ is the exterior of a regular neighbourhood of a knot $K$ in the 3-sphere, and $S$ is an oriented Seifert surface for $K$ viewed as a surface in $X$, then the manifold $M = X \setminus \setminus S$ obtained by cutting $X$ along $S$ admits the structure of a sutured manifold. Here $R_+(\gamma)$ and $R_-(\gamma)$ are the two copies of $S$ in $M$, and $\gamma = \partial X \setminus \setminus \partial S$ is the annulus obtained by cutting $\partial X$ along $\partial S$. 

The notion of taut foliation naturally extends to sutured manifolds \cite{gabai1983foliations}. 

\begin{definition}[Taut foliation on sutured manifolds]
	A transversely oriented codimension-one foliation $\mathcal{F}$ on a sutured manifold $(M, \gamma)$ is \emph{taut}\index{foliation!taut}\index{taut!foliation!on a sutured manifold} if $\mathcal{F}$ is transverse to $\gamma$, tangent to $R(\gamma)$ with the normal direction pointing out of (respectively into) the manifold along $R_+(\gamma)$ (respectively $R_-(\gamma)$), the induced foliation $\mathcal{F}|\gamma$ on $\gamma$ has no Reeb components (i.e. it is a suspension foliation), and each leaf of $\mathcal{F}$ intersects a transverse closed curve or properly embedded arc with endpoints on $R(\gamma)$. 
\end{definition}

	For example, continuing with the knot complement example, if $\mathcal{F}$ is a taut foliation of $X$ that has $S$ as a compact leaf and such that $\mathcal{F}|\partial X$ is a suspension foliation, then the foliation $\mathcal{G} = \mathcal{F} \setminus \setminus S$ obtained by cutting $\mathcal{F}$ along the compact leaf $S$ is a taut foliation of the sutured manifold $ X \setminus \setminus S$. 

	Thurston's theorem saying that compact leaves of taut foliations are norm-minimising naturally extends to the following \cite{gabai1983foliations}. 
	
	\begin{thm}
		If a sutured manifold $(M , \gamma)$ admits a taut foliation $\mathcal{F}$, then the sutured manifold $(M, \gamma)$ is taut, or $M = S^2 \times S^1$ or $S^2 \times [0,1] $ and $\mathcal{F}$ is the product foliation. 
	\end{thm}
	
	Gabai \cite{gabai1983foliations} proved a converse to this theorem when $H_2(M , \gamma) \neq 0$.
	
	\begin{thm}[Gabai]
		Let $M$ be a taut sutured manifold and $H_2(M , \gamma) \neq 0$. Then $(M, \gamma)$ has a finite depth taut foliation $\mathcal{F}$. 
		\label{thm: taut sutured manifolds admit taut foliations}
	\end{thm}

	Note that Theorem \ref{thm:norm-minimising surfaces are leaves} follows from Theorem \ref{thm: taut sutured manifolds admit taut foliations} applied to the sutured manifold $M \setminus \setminus S$. 
	
	\section{The fully marked surface theorem}
	
	\label{sec: fully marked}
	
	Let $M$ be a compact orientable 3-manifold and $\mathcal{F}$ be a taut foliation on $M$.  Recall that a compact properly embedded orientable incompressible and boundary-incompressible surface $S$ in $M$ is  \emph{algebraically fully marked} if the equality in (\ref{eq:Thurston inequality}) happens. The surface $S$ is \emph{positive} (resp. \emph{negative}) \emph{fully marked}\index{fully marked surface!positive/negative} if each component of $S$ is either a leaf of $\mathcal{F}$ whose transverse orientation agrees (resp. disagrees) with that of $\mathcal{F}$, or is transverse to $\mathcal{F}$ except at finitely many saddle tangencies all of which are positive (resp. negative). Here a positive (resp. negative) saddle tangency is a saddle tangency for the induced foliation $\mathcal{F}|S$ such that at the point of tangency the transverse orientations of $S$ and $\mathcal{F}$ agree (resp. disagree). A surface is \emph{fully marked}\index{fully marked surface}\index{surface!fully marked} if it is positive fully marked or negative fully marked. By the Roussarie--Thurston general position \cite{roussarie1974plongements, thurston1986norm}, in a tautly foliated manifold every algebraically fully marked surface is isotopic to a fully marked surface. 
	
	Note that any compact leaf of a taut foliation is fully marked, as is any union of compact leaves that are coherently oriented. The converse is not true as the following example shows. 
	
	\begin{example}
		Let $M$ be a compact orientable 3-manifold that fibers over the circle with fiber a compact orientable surface $S$ and fibration $f \colon M \rightarrow S^1$. Let $\mathcal{F}$ be the foliation by fibers of this fibration. Hence if $d\theta $ is the standard volume form of $S^1 = \mathbb{R}/\mathbb{Z}$, then $\mathcal{F}$ is tangent to the kernel of the closed 1-form $\omega = f^*(d \theta)$. Let  $\omega_1, \cdots, \omega_b$ be closed 1-forms representing a basis  for $H^1(M ; \mathbb{Q})$. Then for $\epsilon_i >0$ small, the 1-form 
		\[ \eta: = \omega + \epsilon_1 \omega_1 + \cdots + \epsilon_b \omega_b\ \]
		is closed and non-singular. Therefore the kernel of $\eta$ defines a foliation $\mathcal{G}$ of $M$ that can be thought as a perturbation of the foliation $\mathcal{F}$ (at the level of plane fields). If at least one of $\{ \epsilon_1, \cdots, \epsilon_b \}$ is irrational, every leaf of this foliation is non-compact. On the other hand, since the plane fields $T \mathcal{F}$ and $T \mathcal{G}$ are homotopic, they have the same Euler class. In particular 
		\[ \langle e(\mathcal{F}) , [S] \rangle  = \langle e(\mathcal{G}) , [S] \rangle. \]
		Since $S$ is algebraically fully marked with respect to $\mathcal{F}$, by the above equality, $[S]$ is algebraically fully marked with respect to $\mathcal{G}$ as well, while $\mathcal{G}$ has no compact leaves. 
	\end{example}

	In fact, every taut foliation of a compact orientable irreducible 3-manifold with no torus and annulus leaves can be $C^0$-approximated by taut foliations that have no compact leaves \cite{bonatti1994feuilles, tsuboi1994hyperbolic}. This shows that in some sense most fully marked surfaces in tautly foliated 3-manifolds are not union of leaves. The fully marked surface theorem shows that a converse is true if we allow to 
	
	\begin{enumerate}
		\item change the foliation while preserving the homotopy class of its tangent plane field, and
		\item change the surface while preserving its homology class. 
	\end{enumerate}

	In some cases, there is a unique norm-minimising surface in the homology class $[S]$, and so the second item above is not necessary. This happens for example if $S$ is a fiber of a fibration of $M$ over $S^1$. In general there might be several isotopy classes of norm-minimising surfaces in a given homology class, and we conjectured in \cite{gabai2020fullymarked} that in general the conclusion of the fully marked surface theorem does not hold without allowing for (2).  Note that if $M$ is a compact orientable irreducible boundary-irreducible atoroidal and anannular 3-manifold and $n$ is an integer, well-known results from normal surface theory show that the number of isotopy classes of orientable incompressible and boundary-incompressible surfaces in $M$ of Euler characteristic $n$ is finite, hence there are finitely many possibilities for the (possibly disconnected) surface in (2) above. See for example \cite{oertel2002existence} for a proof of this result, attributed to Haken. 
	
	
	\section{Contact structures}
	
	\label{sec: contact structures}
	
	Let $M$ be an oriented 3-manifold. A \emph{contact form}\index{contact form} on $M$ is a 1-form $\alpha \in \Omega^1(M)$ such that $\alpha \wedge d\alpha$ is nowhere zero. A plane field $\xi \subset TM$ is a \emph{contact structure}\index{contact structure} if locally it can be defined by a contact 1-form $\alpha$ as $\xi = \ker \alpha$. A contact form is \emph{positive}\index{contact form!positive} if $\alpha \wedge d\alpha$ is a volume form, that is, the orientation of $M$ agrees with that of $\alpha \wedge d\alpha$, and otherwise it is called a \emph{negative contact form}\index{contact form!negative}. A contact manifold\index{contact!manifold} is a pair $(M , \xi)$ where $M$ is an oriented 3-manifold and $\xi$ is a contact structure on $M$. Note that the 1-form $\alpha$ is not part of the data, and if the 1-form $\alpha$ defines a contact structure via $\xi = \ker \alpha$, then for any non-zero smooth function $f \colon M \rightarrow \mathbb{R}$ the 1-form $f \alpha$ defines $\xi$ as well. The standard contact structure on $\mathbb{R}^3$ is defined by $\ker \alpha$ for $\alpha = dz + x dy $. Darboux's theorem states that every contact structure locally looks like the standard contact structure on $\mathbb{R}^3$. 
	
	Given a contact manifold $(M , \xi)$, a knot $K  \subset M$ is called \emph{Legendrian}\index{Legendrian knot}\index{knot!Legendrian} if the tangent vectors satisfy $TK \subset \xi$, i.e. $\alpha (TK) =0$  for the contact 1-form defining $\xi$. The knot $K$ is called \emph{transverse}\index{transverse knot}\index{knot!transverse} if $TK$ is transverse to $\xi$ along the knot $K$, i.e. $\alpha(TK)$ is nowhere vanishing. A \emph{framing}\index{framing} for a knot $K \subset M$ is a trivialisation of the normal bundle of $K$ up to homotopy. The \emph{contact framing}\index{contact!framing}\index{framing!contact}, also called the \emph{Thurston--Bennequin framing}\index{framing!Thurston--Bennequin}\index{Thurston--Bennequin!framing}, of an oriented Legendrian knot $L \subset (M , \xi)$ is defined by the oriented normal of $K$ in $\xi$. 
	
	If $K$ is a nullhomologous knot in $M$ and $S \subset M$ is an embedded oriented surface with $\partial S = K$ then $K$ admits a \emph{Seifert framing}\index{framing!Seifert} defined by the oriented normal of $K$ in $TS_{|K}$. The Seifert framing does not depend on the choice of the bounding surface $S$. Therefore for a nullhomologous Legendrian knot $L \subset (M , \xi)$ we can convert the Thurston--Bennequin framing into an integer $\mathrm{tb(L)}$ which measures the rotation number of the Thurston--Bennequin framing with respect to Seifert framing in the normal plane field to $K$. This number $\mathrm{tb(L)}$ is called the \emph{Thurston--Bennequin number}\index{Thurston--Bennequin!number} of the Legendrian link $L$. 
	
	There is a dichotomy of contact structures into tight and overtwisted structures. Tight contact structures resemble taut or Reebless foliations, and overtwisted contact structures are similar to foliations that have Reeb components.  An \emph{overtwisted disc}\index{overtwisted!disc} for a contact manifold $(M, \xi)$ is an embedded disc $D \subset M$ such that $\partial D = L$ is a Legendrian knot such that the contact framing of $L$ coincides with the framing given by the disc $D$. A contact structure is \emph{overtwisted}\index{overtwisted!contact structure}\index{contact structure!overtwisted} if it has an overtwisted disc, and it is called \emph{tight}\index{contact structure!tight}\index{tight contact structure} otherwise. Unlike taut foliations, a tight contact structure might lift to an overtwisted contact structure via a finite covering. A contact structure $(M , \xi)$ is \emph{universally tight}\index{universally tight}\index{contact structure!universally tight} if the lifted contact structure $(\tilde{M}, \tilde{\xi})$ to the universal cover $\tilde{M}$ of $M$ is tight. 
	
	 The following analogues of Thurston's inequality (\ref{eq:Thurston inequality}) for tight contact structures are due to Eliashberg \cite{eliashberg1992contact} and Bennequin \cite{bennequin1982entrelacements}. 
	
	\begin{thm}[Eliashberg]
		Let $\xi$ be a tight contact structure on an oriented 3-manifold, and let $e(\xi) \in H^2(M ; \mathbb{R})$ be the Euler class of $\xi$. Then for every closed embedded orientable surface $S \subset M$ which is different from $S^2$ the following inequality holds 
		\[ | \langle e(\xi) , [S] \rangle | \leq - \chi(S). \]
		If $S = S^2$ then 
		\[  \langle e(\xi) , [S] \rangle  = 0 . \]
		\label{thm:eliashberg-inequality}
	\end{thm}

	Theorem \ref{thm:eliashberg-inequality} can be proved along lines that philosophically are similar to the proof of Thurston's inequality: namely removing singularities of positive index for the induced foliation $\xi \cap TS$ on $S$ and using an index sum formula to derive the inequality. See \cite{ozbagci-stipsicz} or \cite{geiges-book}.	There are also relative versions of Eliashberg's inequality (Theorem \ref{thm:eliashberg-inequality}) for embedded orientable surfaces whose boundary is either a transverse or a Legendrian link in a tight contact manifold. Let $S \subset (M, \xi)$ be an embedded orientable surface such that $\Gamma = \partial S$ is transverse to $\xi$. Assume that the orientations of $\Gamma$, $\xi$, and $M$ are related in the following way: when $\Gamma$ is oriented as $\partial S$, at any point $x \in \Gamma$ the orientation of the plane $\xi_x$ together with the orientation of $T_x\Gamma$ gives the orientation of $T_xM$. The relative Euler class\index{relative Euler class!of a contact structure} $e(\xi)$ is defined as follows: Given a vector field $X$ along $\Gamma$ that generates the line field $\xi \cap T(S)$, the number $\langle e(\xi) , [S] \rangle$ is the obstruction for the extension of $X$ to $S$ as a vector field in $\xi$. In particular, the relative Euler number $e(\xi)$ is an element of $H^2(N , \partial N)$ where $N$ is the complement of a regular neighbourhood of $\Gamma = \partial S$ in $M$. 
	
	\begin{thm}[Eliashberg]\index{Eliashberg inequality}
		Let $\xi$ be a tight contact structure on an oriented 3-manifold. If $S$ is an embedded orientable surface with boundary transverse to $\xi$ then we have the following inequality 
		\[ | \langle e(\xi) , [S] \rangle | \leq - \chi(S), \]
		where $e(\xi)$ is the relative Euler class of $\xi$.
		\label{thm:eliashberg-transverse knot}
	\end{thm}

	Now let $L$ be a nullhomologous oriented Legendrian link in a contact manifold $(M, \xi)$, and $S$ be an embedded oriented surface with boundary $L$. The \emph{rotation number}\index{rotation number} of $L$, denoted by $\mathrm{rot}_S(L)$, is defined as the relative Euler number of $\xi_{|S}$ with the trivialisation of $\xi$ along $\partial S$ given by the tangents of $L$. The rotation number $\mathrm{rot}_S(L) \in \mathbb{Z}$ in general depends on the choice of $S$. The rotation number also depends on the orientation of $L$ and changes sign when the orientation of $L$ is reversed. 
	
	\begin{thm}[Eliashberg]\index{Eliashberg inequality}
		The contact 3-manifold $(M , \xi)$ is tight if and only if for all embedded oriented $S \subset M$ with $\partial S$ Legendrian we have 
		\[ \mathrm{tb}(L) + |\mathrm{rot}_S(L)| \leq - \chi(S). \]
		\label{thm:eliashberg-Legendrian knot}
	\end{thm}

	Since an overtwisted disc has $\chi(D) =1$ and $\mathrm{tb}(L)=0$, a contact structure satisfying the inequality in Theorem \ref{thm:eliashberg-Legendrian knot} for all $S$ with Legendrian boundary must be tight. Bennequin proved the above inequality for the standard contact structure $\xi_{\mathrm{std}}$ on $\mathbb{R}^3$. In particular, he showed that $(\mathbb{R}^3, \xi_{\mathrm{std}})$ is tight. Every Legendrian curve in a contact manifold $(M , \xi)$ can be $C^0$ approximated by transverse curves, and using this Theorem \ref{thm:eliashberg-Legendrian knot} can be deduced from Theorem \ref{thm:eliashberg-transverse knot}. It follows that if a contact structure $(M , \xi)$ satisfies the inequality in Theorem \ref{thm:eliashberg-transverse knot} for every such $S$ with $\partial S$ transverse to $\xi$, then $\xi$ is tight. See Eliashberg and Thurston \cite[Section 3.3]{eliashberg1998confoliations} for a discussion of Inequality (\ref{eq:Thurston inequality}) for contact structures and \emph{confoliations}\index{confoliation} (which is a hybrid structure between foliations and contact structures).  

	Eliashberg and Thurston \cite[Theorem 2.4.1]{eliashberg1998confoliations} showed that taut foliations can be approximated by tight contact structures. 
	
	\begin{thm}[Eliashberg--Thurston]
		Suppose that $\mathcal{F}$ is a codimension-one $C^2$ foliation of a 3-manifold that is different from the product foliation on $S^2 \times S^1$ by leaves $S^2 \times \text{point}$. Then $\mathcal{F}$ can be $C^0$-approximated by a pair $\xi_+$ and $\xi_-$ of positive and negative contact structures. 
	\end{thm}

	If $(M , \xi)$ is a closed contact 3-manifold, a \emph{weak symplectic filling}\index{weak symplectic filling} of $(M, \xi)$ is a compact symplectic 4-manifold $(W, \omega)$ with $\partial W = M$ as oriented manifolds, and such that $\omega |\xi$ is positive definite everywhere. We say that $(M , \xi)$ is \emph{weakly symplectically fillable}\index{weakly symplectically!fillable} if such a $(W , \omega)$ exists. A contact 3-manifold is \emph{weakly symplectically semi-fillable}\index{weakly symplectically!semi-fillable} if it is a connected component of a weakly symplectically fillable manifold.
	Eliashberg and Thurston \cite{eliashberg1998confoliations} showed that contact structures $C^0$-close to taut foliations are weakly symplectically semi-fillable and universally tight \cite[Corollaries 3.2.5 and 3.2.8]{eliashberg1998confoliations}. Later it was shown that weakly symplectically semi-fillable contact structures are weakly symplectically fillable \cite{eliashberg2004few, etnyre2004symplectic}. Additionally, Bowden \cite{bowden2016approximating} and Kazez and Roberts \cite{kazez2015approximating} generalised the work of Eliashberg and Thurston to $C^0$ foliations.  
	
	Since the Euler class is invariant under homotopy of plane fields, the above results of Eliashberg and Thurston imply that for $M$ a closed oriented 3-manifold, any cohomology class $a \in H^2(M ; \mathbb{Z})$ which is realised as the Euler class of a taut foliation on $M$ is also realised as the Euler class of a tight contact structure on $M$. Moreover, there are cases that a cohomology class is realised by tight contact structures but not by taut foliations. For example $S^3$ has a unique positive tight contact structure up to isotopy \cite{bennequin1982entrelacements, eliashberg1992contact}, but no taut foliation \cite{novikov1965topology}. In \cite{yazdi2020thurston} the author asked if the Euler class one conjecture holds for tight contact structures. 
	
	\begin{question}
		Let $M$ be a closed orientable hyperbolic 3-manifold with first Betti number at least one. Is every integral class $a \in H^2(M ; \mathbb{R})$ of dual norm one and satisfying the parity condition realised as the Euler class of a tight contact structure?
		\label{que:euler class one for tight contact structures}
	\end{question}
	
	With Steven Sivek \cite{sivek2023thurston} we showed that the counterexample cohomology classes in \cite{yazdi2020thurston} \emph{are} realised by possibly negative tight (in fact weakly symplectically fillable) contact structures. Recently, Yi Liu has proved the following remarkable result \cite{liu2024euler}.
	
	\theoremstyle{theorem}
	\newtheorem*{Liu-contact}{Theorem \ref{thm:Liu-contact}}
	\begin{Liu-contact}[Liu]
		For every closed hyperbolic 3–manifold $M$, there exists some connected finite cover $\tilde{M}$ of $M$, and some even lattice point $\tilde{w} \in H^2(\tilde{M}; \mathbb{R})$ of dual Thurston norm one such that $\tilde{w}$ is not the real Euler class of any weakly symplectically fillable contact structure on $\tilde{M}$.
	\end{Liu-contact}	
	
	In particular, his result produces many new counterexamples to the Euler class one conjecture for taut foliations.

	\theoremstyle{theorem}
	\newtheorem*{Liu-corollary}{Theorem \ref{cor:Liu-taut-foliation}}
	\begin{Liu-corollary}[Liu]
		Every closed hyperbolic 3-manifold has a finite cover for which the Euler class one conjecture (for taut foliations) does not hold. 
	\end{Liu-corollary}

	We now explain some of the ingredients in Liu's proof of Theorem \ref{thm:Liu-contact}, following \cite{liu2024euler}. The first idea for proving that weakly symplectically fillable contact structures with a given Euler class do not exist is a non-vanishing result. Liu \cite{liu2024euler} proves the following, based on a non-vanishing result due to Ozsv\'{a}th and Szab\'{o} \cite[Theorem 4.2]{ozsvath2004holomorphic}. In the following $\widehat{\mathrm{HF}}$ stands for the (hat flavour of) \emph{Heegaard Floer homology}\index{Heegaard Floer homology} introduced by Ozsv\'{a}th and Szab\'{o}, which is an abelian group graded by $\mathrm{Spin}^c$ structures on the manifold. Here we consider a \emph{$\mathrm{Spin}^c$ structure}\index{$\mathrm{Spin}^c$ structure} as a non-vanishing vector field on $M$, where two such vector fields are equivalent if one can be homotoped to the other one outside of a 3-ball through non-vanishing vector fields.  Therefore $ \mathrm{rank}_\mathbb{Z} \widehat{\mathrm{HF}}(- M, \mathfrak{s}_\xi) $ denotes the rank of the Heegaard Floer homology group of the manifold $-M$, that is $M$ with the opposite orientation, in the $\mathrm{Spin}^c$ grading $\mathfrak{s}_\xi$. 
	
	\begin{prop}[Non-vanishing criterion]
		If $(M , \xi)$ is an oriented closed contact 3–manifold that is weakly symplectically fillable, then 
		\[ \mathrm{rank}_\mathbb{Z} \widehat{\mathrm{HF}}(- M, \mathfrak{s}_\xi) > 0. \]
		Here, $\mathfrak{s}_\xi$ denotes the canonical $\mathrm{Spin}^c$ structure of $\xi$, which is represented by any nowhere vanishing vector field transverse to and agreeing with its prescribed transverse orientation.
	\end{prop} 

	When $M= M_f$ is the mapping torus of a pseudo-Anosov\index{pseudo-Anosov!map} map $f \colon S \rightarrow S$, deep results of Cotton-Clay \cite{cotton2009symplectic}, and Kutluhan--Lee--Taubes \cite{kutluhan2020hfI,kutluhan2020hfII,kutluhan2020hfIII,kutluhan2021hfIV,kutluhan2021hfV}, and Lee--Taubes \cite{lee2012periodic} identify next-to-top terms in Heegaard Floer homology of the mapping torus and the \emph{Periodic Floer homology}\index{Periodic Floer homology} of the suspension flow. Using the above non-vanishing criterion and the mentioned connection, Liu \cite{liu2024euler} obtains the following non-realisability criterion for weakly symplectically fillable contact structures. 
	
	\begin{prop}[Non-realisability criterion]
		Let $S$ be an oriented connected closed surface of genus $g \geq 3$, and
	$f \colon S \rightarrow S$ be a pseudo-Anosov homeomorphism. Let $M_f$ be the mapping torus of $f$, and denote by $e_f \in H^2(M_f ; \mathbb{R})$ the real Euler class of the associated fibration of $M_f$ over $S^1$.	If $a \in H^2(M_f ; \mathbb{R})$ is an integral lattice point satisfying $\langle a , [S] \rangle =1$, and if the Poincar\'{e} dual $\mathrm{PD}(a) \in H_1(M_f ; \mathbb{R})$ is not represented by any $1$-periodic trajectory of the suspension flow, then the even lattice point $e_f + 2a$ is not the real Euler class of any weakly symplectically fillable contact structure on $M_f$, nor is $e_f - 2a$.
	\end{prop}

	Given a fibered 3-manifold $N$ with pseudo-Anosov monodromy $f$, the \emph{Fried cone of homology directions}\index{Fried cone of homology directions} $\mathcal{C}_f$ is a polyhedral cone in $H_1(N)$, defined as the closure of homology classes of periodic orbits of the suspension flow of $f$. It is known that the Fried cone in $H_1(N)$ is dual to the vertex of the dual Thurston norm ball in $H^2(N)$ that is the Euler class $e_f$ of the fibration of $N$ over $S^1$. See \cite{fried1979fibrations}. Having the above non-realisability criterion in mind, given a closed hyperbolic 3-manifold $M$, we are interested in finding a covering $\tilde{M}$ of $M$ that fibers over the circle with monodromy say $f$ and such that the Fried cone $\mathcal{C}_f \in H_1(\tilde{M})$ contains no 1-periodic trajectory of the suspension flow. Liu proves Theorem \ref{thm:Liu-contact} by constructing a fibered finite cover $\tilde{M}$ of $M$ with $b_1(\tilde{M}) \geq 2$, and such that some boundary face of the Fried cone $\mathcal{C}_f$ has a scaling-invariant open dense subset, such that no periodic trajectory represents any rational homology class therein. His proof of this statement relies on a virtual construction and uses the virtual compact specialisation\index{virtual compact specialisation theorem} of closed hyperbolic 3–manifold groups due to Agol \cite{agol2013virtual} and Wise \cite{wise2012riches}. We refer the reader to \cite{liu2024euler} for this construction. 

	\section{Quasigeodesic and pseudo-Anosov flows}
	
	\label{sec: flows}

	A flow on a closed 3-manifold $M$ is called \emph{quasigeodesic}\index{flow!quasigeodesic}\index{quasigeodesic flow} if the lifted flow lines to the universal cover $\tilde{M}$ of $M$ are quasigeodesics with respect to the lift of a metric on $M$ to $\tilde{M}$. By compactness of $M$, the property of being quasigeodesic does not depend on the choice of metric on $M$. A theorem of Zeghib \cite{zeghib1993feuilletages} states that there is no flow on a closed hyperbolic manifold such that all flow lines are geodesics. However, quasigeodesic flows exist on many hyperbolic 3-manifolds. When $M$ is hyperbolic, any quasigeodesic in the universal cover $\tilde{M} \cong \mathbb{H}^3$ is of bounded distance to a geodesic. This property is very useful in understanding quasigeodesic flows on hyperbolic 3-manifolds. 
	
	Cannon and Thurston \cite{cannon2007group} studied the first examples of quasigeodesic flows on fibered hyperbolic 3-manifolds. Let $S$ be a closed orientable surface of genus $g \geq 2$, and $\phi \colon S \rightarrow S$ be a  pseudo-Anosov homeomorphism of $S$. Denote the mapping torus of $\phi $ by $M_\phi$:
	\[ M_\phi := (S \times \mathbb{R}) / \sim, \]
	with the equivalence relation $(x , n) \sim (f(x) , n+1)$ for every $x \in S$ and $n \in \mathbb{Z}$. Therefore $M_\phi$ fibers over $S^1$ with fiber $S$. By Thurston's hyperbolisation theorem, $M_\phi$ is hyperbolic. The flow lines $x \times \mathbb{R}$ on $S \times \mathbb{R}$ descend to a flow $\Phi$ on $M_\phi$. Cannon and Thurston showed that $\Phi$ is a quasigeodesic flow. Note that $\Phi$ is transverse to the depth $0$ foliation of $M_\phi$, which is the fibration of $M_\phi$ over the circle. 
	
	A flow on a closed 3-manifold is \emph{pseudo-Anosov}\index{flow!pseudo-Anosov}\index{pseudo-Anosov!flow} if it is locally modeled on the suspension flow of a pseudo-Anosov surface homeomorphism, even though globally the flow need not be a suspension flow, see \cite{mosher1992dynamical} for the precise definition. Mosher \cite{mosher1examples} produced examples of quasigeodesic flows transverse to a class of depth one foliations in hyperbolic 3-manifolds. Mosher \cite{mosher1996laminations}, following Gabai, showed that given a finite depth taut foliation $\mathcal{F}$ of a hyperbolic 3-manifold there is a pseudo-Anosov flow \emph{almost transverse}\index{pseudo-Anosov!flow!almost transverse to a foliation} to $\mathcal{F}$. Almost transverse means that the flow is transverse to $\mathcal{F}$ after blowing-up a finite number of closed orbits of the flow. See \cite[Section 3.5]{mosher1996laminations} for the precise definition. In particular the Euler class of the oriented normal plane field to the flow is equal to the Euler class of the oriented tangent plane field to the foliation. Fenley and Mosher \cite{fenley2001quasigeodesic} proved that given a finite depth taut foliation, any pseudo-Anosov flow almost transverse to the foliation is quasigeodesic as well. In particular, the flows constructed by Mosher \cite{mosher1996laminations} are quasigeodesic. Mosher's construction of pseudo-Anosov flows almost transverse to finite depth taut foliations is largely unwritten, although the first part is available at \cite{mosher1996laminations}. However, Landry and Tsang are in the process of writing down a proof of this result. See \cite{landry2024endperiodic}. Mosher \cite{mosher1992dynamical, mosher1992dynamicalII} proved that the Euler class of every quasigeodesic pseudo-Anosov flow has dual Thurston norm at most one. He showed this by proving an efficient intersection theorem between embedded incompressible surfaces and the flow, and then comparing index sum formulae. Moreover, analogues of Inequality (\ref{eq:Thurston inequality}) hold for Euler classes of pseudo-Anosov flows, and also for quasigeodesic flows. See Section \ref{sec:universal circle} for more on this, which goes via universal circle actions. 
	
	Since vertices of the unit ball of the dual Thurston norm are realised by Euler classes of taut foliations (Theorem \ref{thm:Gabai-euler class one for vertices}), it follows that they are also realised as the Euler classes of quasigeodesic pseudo-Anosov flows.  In \cite{yazdi2020thurston}, the author asked the following question. 
	
	\begin{question}
		Let $M$ be a closed orientable hyperbolic 3-manifold with first Betti number at least one. Is every integral class $a \in H^2(M ; \mathbb{R})$ of dual norm one and satisfying the parity condition realised as the Euler class of a pseudo-Anosov (respectively quasigeodesic) flow?
		\label{que:euler class one for tight contact structures}
	\end{question}

	\section{Actions on the circle}
	
	\label{sec: actions on cricle}
	
	\subsection{Milnor--Wood inequality}
	
	The group $\mathrm{PSL}(2, \mathbb{R})$ acts on $S^1 = \mathbb{R} \cup \{ \infty \}$ by projective transformations\index{projective transformations} as 
	\[ \begin{bmatrix}
		a & b \\
		c & d
	\end{bmatrix} \colon  x  \longmapsto \frac{ax+ b}{cx+d}. \] Let $S$ be a closed orientable surface and
	\[ \rho \colon \pi_1(S) \rightarrow \mathrm{PSL}(2, \mathbb{R}) \]
	be a representation\index{representation}. The representation $\rho$ defines an $S^1$-bundle over $S$ via the suspension construction, and we denote the Euler class of this $S^1$-bundle by $e(\rho) \in H^2(S ; \mathbb{Z})$. Milnor \cite{milnor1958existence} showed that 
	\begin{eqnarray}
		|\langle e(\rho) , [S] \rangle | \leq \chi_-(S),
		\label{Minor-Wood inequality}
	\end{eqnarray}
	where $[S] \in H_2(S ; \mathbb{Z})$ is the fundamental class of $S$. He also showed that for a closed orientable surface $S$, every integral class $a \in H^2(S ; \mathbb{Z})$ satisfying Inequality (\ref{Minor-Wood inequality}) is realised as the Euler class of a representation into $\mathrm{PSL}(2, \mathbb{R})$.  Wood \cite{wood1971bundles} generalised Inequality (\ref{Minor-Wood inequality}) to representations 
	\[ \rho \colon \pi_1(S) \rightarrow \mathrm{Homeo}^+(S^1), \]
	where $\mathrm{Homeo}^+(S^1)$ is the group of orientation-preserving homeomorphisms of $S^1$. This is now known as the \emph{Milnor--Wood inequality}\index{Milnor--Wood inequality}. 
	
\begin{thm}[Milnor--Wood inequality]
	Let $S$ be a closed orientable surface of genus $g$. Every homomorphism $\rho \colon \pi_1(S) \rightarrow \mathrm{Homeo}^+(S^1)$ satisfies 
	\[  |\langle e(\rho) , [S] \rangle | \leq \chi_-(S). \]
\end{thm}

	Now let $M$ be a closed orientable 3-manifold, and 
	\[ \rho \colon \pi_1(M) \rightarrow \mathrm{Homeo}^+(S^1) \]
	be a representation. By the Milnor--Wood inequality, for every embedded orientable surface $S$ in $M$ Inequality (\ref{Minor-Wood inequality}) holds. In other words, the dual Thurston norm of $e(\rho) \in H^2(M ; \mathbb{Z})$ is at most one. In \cite{yazdi2020thurston}, the author asked the following question.
	
	\begin{question}
		Let $M$ be a closed orientable hyperbolic 3-manifold with first Betti number at least one. Is every integral class $a \in H^2(M ; \mathbb{R})$ of dual norm one realised as the Euler class of a representation $\rho \colon \pi_1(M) \rightarrow \mathrm{Homeo}^+(S^1)$?
		\label{que:euler class one for actions}
	\end{question}

	We can also ask for variations where the cohomology class $a \in H^2(M ; \mathbb{R})$ has norm at most one, or ask for representations of a particular regularity class, or restrict to a class of 3-manifolds.  
	
	\subsection{Bounded cohomology} Bounded cohomology was introduced by Johnson \cite{johnson1972cohomology} and Gromov \cite{gromov1982volume}. The natural framework for the Milnor--Wood inequality is in the context of bounded cohomology, as developed by Ghys \cite{ghys1987groupes}. In this section we first recall the definition of group cohomology, and then make the necessary adjustments to define bounded cohomology of groups. Ghys' survey \cite{ghys2001groups} is an excellent reference for the material in this subsection and the next. Let $\Gamma$ be a (discrete) group, and $E\Gamma$ be the semi-simplicial set whose vertices are elements of $\Gamma$ and whose $k$-simplices are $(k+1)$-tuples of elements of $\Gamma$ for any integer $k \geq 0$. The $i$-th face of a simplex $(\gamma_0 , \gamma_1 , \cdots, \gamma_k)$ is $(\gamma_0 , \cdots , \hat{\gamma_i}, \cdots \gamma_k)$ where $\hat{\gamma_i}$ indicates that $\gamma_i$ is omitted. Then $E \Gamma$ is contractible since it is the full simplex over the set $\Gamma$. There is a natural simplicial action of $\Gamma$ on $E \Gamma$ induced by the left action of $\Gamma$ on itself 
	\[ \gamma \cdot (\gamma_0 , \gamma_1, \cdots , \gamma_k) := (\gamma \gamma_0 , \gamma \gamma_1, \cdots, \gamma \gamma_k). \]
	This simplicial action is free and permutes the set of $k$-simplices. The quotient of $E \Gamma $ by the action of $\Gamma$ is denoted by $B \Gamma$ and is a \emph{classifying space}\index{classifying space} for $\Gamma$, in the sense that $\pi_1(B \Gamma) \cong \Gamma$ and that higher homotopy groups of $B \Gamma$ are trivial since $\pi_n(B \Gamma) \cong \pi_n(E \Gamma) \cong \{ 0 \}$ for $n \geq 2$. The cohomology groups of $\Gamma$ are defined as the cohomology groups of $B \Gamma$. Algebraically, this can be described as follows. Define a \emph{$k$-cochain} of $\Gamma$ with coefficients in some abelian group $A$ as a map $c \colon \Gamma^{k+1} \rightarrow  A$ that is invariant under the action of $\Gamma$, that is  $c(\gamma_0 , \cdots, \gamma_k) = c(\gamma \gamma_0 , \cdots , \gamma \gamma_k)$. Such cochains are called \emph{homogeneous}\index{cochain!homogeneous}. Let $C^k(G , A)$ be the set of $k$-cochains, which is an abelian group. There is a natural coboundary map 
	\[ d_k \colon C^k(\Gamma , A) \rightarrow C^{k+1}(\Gamma , A) \] 
	defined as 
	\[ d_k c (\gamma_0 , \cdots, \gamma_{k+1}):= \sum_{i=0}^{k+1} (-1)^i c(\gamma_0 , \cdots, \hat{\gamma_i}, \cdots, \gamma_{k+1}).\]
	It is easy to check that $d_{k+1} \circ d_k = 0$. The cohomology group $H^k(\Gamma , A)$ is defined as the quotient of cocycles (i.e. the kernel of $d_{k}$) by coboundaries (i.e. the image of $d_{k-1}$). 

	Any homogeneous cochain $c \colon \Gamma^{k+1} \rightarrow A$ can be equivalently described by the cochain $\bar{c} \colon \Gamma^k \rightarrow A$ defined as 
	\[ c(\gamma_0, \cdots , \gamma_k) = \bar{c}(\gamma_0^{-1} \gamma_1, \cdots, \gamma_0^{-1} \gamma_k). \]
	Such cochains $\bar{c}$ are called \emph{inhomogeneous}\index{cochain!inhomogeneous}. Conversely an inhomogeneous $k$-cochain defines a corresponding homogeneous $k$-cochain.
	
	Now let $\Gamma$ be a group as before, and $A = \mathbb{Z}$ or $\mathbb{R}$. Let $C^k_b(\Gamma , A)$ be the set of homogeneous $k$-cochains that are bounded (as real-valued functions). Then $C^k_b(\Gamma , A)$ is a subgroup of $C^k(\Gamma , A)$, and the coboundary $d_k$ of a bounded $k$-cochain is again a bounded $(k+1)$-cochain. The bounded cohomology groups\index{bounded!cohomology group} $H^*_b(\Gamma , A)$ are defined as cohomology groups of the complex $(C^k_b(\Gamma, A), d_k)$. There is a natural inclusion 
	\[ C^k_b(\Gamma, A) \rightarrow C^{k}(\Gamma , A) \]
	and this induces a homomorphism 
	\[ H^k_b(\Gamma, A) \rightarrow H^k(\Gamma , A)\] 
	called the \emph{comparison map}\index{comparison map}. In general the comparison map is neither injective nor surjective. However, when $\Gamma$ is Gromov-hyperbolic\index{Gromov-hyperbolic}, the comparison map is surjective for all $k \geq 2$, see \cite{mineyev2001straightening}. 
	
	\subsection{Milnor--Wood inequality revisited}
	Let $\widetilde{\mathrm{Homeo}^+}(S^1)$ be the group of orientation-preserving homeomorphisms of $\mathbb{R}$ that commute with integral translations, i.e.
		\[ \widetilde{\mathrm{Homeo}^+}(S^1) := \{ f \in \mathrm{Homeo}^+(\mathbb{R}) | f (x+1) = f(x)+1 \text{ for every } x \text{ in } \mathbb{R}\}. \]
		There is a short exact sequence 
		\begin{eqnarray}
				0 \rightarrow \mathbb{Z} \rightarrow \widetilde{\mathrm{Homeo}^+}(S^1) \xrightarrow{\pi} \mathrm{Homeo}^+(S^1) \rightarrow 0, 
				\label{eqn: central extension}
			\end{eqnarray} 
		where $\mathbb{Z}$ is identified with the subgroup of integral translations in $\widetilde{\mathrm{Homeo}^+}(S^1)$. In particular this is a central extension of $\mathrm{Homeo}^+(S^1)$ by $\mathbb{Z}$.
	Let $s \colon \mathrm{Homeo}^+(S^1) \rightarrow  \widetilde{\mathrm{Homeo}^+}(S^1)$ be a set-theoretic section (not necessarily a homomorphism). Define an inhomogeneous 2-cochain on $\mathrm{Homeo}^+(S^1)$ with values in $\mathbb{Z}$ by 
	\[ \overline{c}(f, g) := s(f \circ g)^{-1} s(f) s(g).  \]
	Note that the projection under the map $\pi$ of the right hand side onto $\mathrm{Homeo}^+(S^1)$ is trivial, and therefore we can identify it with an element of $\mathbb{Z}$. Then it is a computation to see that $c$ is a cocycle. Moreover, the cohomology class of $c$ does not depend on the choice of the section $s$, and is called the \emph{Euler class}. The Euler class is denoted by $e \in H^2(\mathrm{Homeo}^+(S^1) , \mathbb{Z})$, and it is known that it generates $H^2(\mathrm{Homeo}^+(S^1) , \mathbb{Z})$.  The preceding discussion can be generalised to see that if $A$ is an abelian group then every central extension\index{central extension} 
	\[ 0 \rightarrow A \rightarrow \hat{\Gamma} \rightarrow \Gamma \rightarrow 0 \]
	defines an element of $H^2(\Gamma , A)$, and conversely elements of $H^2(\Gamma , A)$ correspond to isomorphism classes of central extensions of $\Gamma$ by $A$. 
	
	Now for the central extension (\ref{eqn: central extension}) we can choose a canonical section $s$ by requiring that $s(f)(0) \in [0 , 1)$. It is easy to check that for this section the value
	\[ \overline{c}(f, g)(0) =  \big(s(f \circ g)^{-1} s(f) s(g) \big) (0)\]
	lies in $[0, 2)$. Since this is an integer, it must be one of the numbers $0$ or $1$. Therefore, the cocycle $\overline{c}$ only takes values in $\{ 0, 1\}$; in particular it defines a bounded cocycle. The bounded cohomology class of this bounded cocycle does not depend on the choice of the origin $0$, and is called the \emph{bounded Euler class}\index{bounded!Euler class}\index{Euler class!bounded}, and is denoted by $e_b \in H^2_b(\mathrm{Homeo}^+(S^1) , \mathbb{Z})$. Thinking of the bounded Euler class with real coefficients instead, we obtain a \emph{real bounded Euler class} $e_b^\mathbb{R} \in H^2_b(\mathrm{Homeo}^+(S^1) , \mathbb{R})$. 
	
	Given a bounded $k$-cochain $c \in C^k(\Gamma, \mathbb{R})$, define its norm $||c||$ as the supremum of the value of $c(\gamma_0 , \gamma_1, \cdots, 
	\gamma_k)$. Define the \emph{norm} of a bounded cohomology class as the infimum of the norm of any of its cocycle representatives. This is in general a semi-norm, i.e. the norm of a non-trivial class could be $0$. The following can be thought as the Milnor--Wood inequality in the language of bounded cohomology.

	\begin{thm}
		The real bounded Euler class $e_b^\mathbb{R} $ has norm $\frac{1}{2}$. 
	\end{thm}


	See \cite{ghys2001groups} for the proof. Given an action\index{action}\index{action!of a group}\index{group action} $\rho \colon \Gamma \rightarrow \mathrm{Homeo}^+(S^1)$ of a group $\Gamma$ on the circle\index{action!on the circle}\index{group action!by homeomorphisms of the circle}, we can define a bounded Euler class $\rho^*(e_b)$ by pulling back the bounded Euler class of $H^2_b(\mathrm{Homeo}^+(S^1) , \mathbb{Z})$. Ghys has characterised which bounded second cohomology classes are obtained from circle actions \cite{ghys1987groupes, ghys2001groups}. 

	\begin{thm}[Ghys]
		Let $\Gamma$ be a countable group and $c$ be an element of $H^2_b(\Gamma , \mathbb{Z})$. Then there exists a homomorphism $\rho \colon \Gamma \rightarrow \mathrm{Homeo}^+(S^1)$ such that $\rho^*(e_b) = c$ if and only if $c$ can be represented by a cocycle that takes only values $0$ and $1$. 
	\end{thm}
	
	In calculating the group cohomology of a group $G$, we can use any classifying space for $G$. Namely any path-connected space $X$ with $\pi_1(X) \cong G$ and having trivial higher homotopy groups satisfies $H^k(G ; \mathbb{R}) \cong H^k(X ; \mathbb{R})$. If $M$ is an aspherical 3-manifold (for example if $M$ is hyperbolic) then $M$ is a classifying space for $\pi_1(M)$, since the universal cover of $M$ is contractible. Therefore there is a natural isomorphism $H^k(\pi_1(M), \mathbb{R}) \cong H^k(M , \mathbb{R})$. So Question \ref{que:euler class one for actions} could be paraphrased as whether for $M$ a closed hyperbolic 3-manifold with first Betti number at least one, every integral cohomology class $c \in H^2(\pi_1(M), \mathbb{R}) \cong H^2(M , \mathbb{R})$ of dual Thurston norm one has a cocycle representative taking only values $0$ and $1$. 

	\subsection{Actions on the universal circle}
	\label{sec:universal circle}
	
	The uniformisation theorem\index{uniformisation theorem} states that given a Riemann surface $S$ there is a Riemannian metric in the same conformal class that is of constant curvature. In particular, if $\chi(S) <0 $, then $S$ admits a complete hyperbolic metric\index{hyperbolic!metric}. Let $\mathcal{F}$ be a foliation (or lamination\index{lamination!by surfaces}) of a 3-manifold $M$ by surfaces, and $g$ be a metric on $M$. Candel \cite{candel2003foliations} proved a parametric version of the uniformisation theorem and showed that, assuming certain necessary condition, there is a metric on $M$ in the same conformal class as $g$ such that the induced metric on each leaf of $\mathcal{F}$ is a complete hyperbolic metric. 
	
	\begin{thm}[Candel's uniformisation theorem]
		Let $\Lambda$ be a Riemann surface lamination such that for every invariant transverse measure\index{invariant!transverse measure} $\mu$ we have $\chi(\mu)<0$. Then there is a continuously varying leafwise metric on $\Lambda$ where the leaves are locally isometric to the hyperbolic plane. 
	\end{thm}	
	
	When $\mathcal{F}$ is a taut foliation of an atoroidal 3-manifold, every invariant transverse measure has negative Euler characteristic, and so by Candel's theorem, $M$ admits a metric such that the induced metric on every leaf is hyperbolic. In this case, Thurston \cite{thurston1998three} and Calegari and Dunfield \cite{calegari2003laminations} constructed a faithful action\index{action!faithful} 
	\[ \rho_{\mathrm{univ}} \colon \pi_1(M) \rightarrow \mathrm{Homeo}^+(S^1), \]
	called the \emph{universal circle action}\index{universal circle!action}\index{universal circle}. Given a leaf $L$ of $\mathcal{F}$, there is an action of $\pi_1(L)$ on the ideal boundary $\partial \mathbb{H}^2 \cong S^1_\infty$ of its universal cover $\tilde{L} \cong \mathbb{H}^2$, and the universal circle action collates such actions for different leaves $L$ of the foliation $\mathcal{F}$. The Euler class of the action $\rho_\mathrm{univ}$ is defined as the Euler class of the associated $S^1$-bundle over $M$, constructed via the suspension construction. It is known that the Euler class of $\rho_\mathrm{univ}$ is equal to the Euler class of the taut foliation $\mathcal{F}$. See Boyer and Hu \cite{boyer2019taut} for a proof of this fact. Therefore, Question \ref{que:euler class one for actions} is weaker than the original Euler class one conjecture for taut foliations. In particular, by Gabai's Theorem \ref{thm:Gabai-euler class one for vertices} (Euler class one for vertices) and the universal circle action construction, every vertex of the unit ball of the dual Thurston norm is realised as the Euler class of an action of $\pi_1(M)$ by orientation-preserving homeomorphisms on $S^1$. As a first step towards Question \ref{que:euler class one for actions}, it would be interesting to find a purely algebraic proof of this latter fact.
	
	 There are also constructions of universal circle actions for pseudo-Anosov flows, and for quasigeodesic flows on closed hyperbolic 3-manifolds. Calegari and Dunfield \cite{calegari2003laminations} showed that if $M$ is a closed atoroidal 3-manifold with a pseudo-Anosov flow $\Phi$, then there is an action of $\pi_1(M)$ on a circle $S^1_{\mathrm{univ}}$ that preserves a pair of invariant laminations\index{invariant! laminations}. For a quasigeodesic flow on a hyperbolic 3-manifold, the leaf space\index{flow!leaf space of}\index{leaf!space} $P_\Phi$ of the lifted flow $\tilde{\Phi}$ to the universal cover $\tilde{M}$ is homeomorphic to $\mathbb{R}^2$. Moreover, Calegari \cite[Remark 5.2]{calegari2006universal} showed that there is a compactification of $P_\Phi$ to a closed disc $D_\Phi  = P_\Phi \cup S^1_\mathrm{univ}$ such that the natural action of $\pi_1(M)$ on $P_\Phi$ extends to the boundary $\partial D_\Phi =  S^1_\mathrm{univ}$. Additionally, the action of $\pi_1(M)$ on the circle $S^1_{\mathrm{univ}}$ 
	\[ \rho_\Phi \colon \pi_1(M) \rightarrow \mathrm{Homeo}^+(S^1_{\mathrm{univ}})  \]
	preserves a pair of invariant laminations \cite[Theorem A]{calegari2001leafwise}, and the Euler class of the associated $S^1$-bundle over $M$ is equal to the Euler class of the normal plane field to the flow \cite[Lemma 6.4]{calegari2001leafwise}. We would like to stress that while we denoted several circles by the same symbol $S^1_{\mathrm{univ}}$, the relation between the various universal circle actions above, for taut foliations, pseudo-Anosov flows, and quasigeodesic flows, is subject of current research. See for example \cite{huang2024depth, landry2024simultaneous}.

	\section{Virtual Euler class one conjecture}
	
	\label{sec: virtual}
	
	Let $M$ be a closed orientable 3-manifold and $p \colon \tilde{M} \rightarrow M$ be a finite covering map. Denote by $p^* \colon H^2(M; \mathbb{R}) \rightarrow H^2(\tilde{M}, \mathbb{R})$ the induced map on cohomology. It follows from deep results of Gabai \cite{gabai1983foliations} that $p^*$ preserves the dual Thurston norm, see \cite[Proposition 2.20]{yazdi2020thurston}. Now let $a \in H^2(M ; \mathbb{R})$ be an integral point satisfying the parity condition. If $a$ is equal to the Euler class of some taut foliation $\mathcal{F}$ on $M$, then $p^*(a)$ is equal to the Euler class of the lifted foliation $\tilde{\mathcal{F}}$ on $\tilde{M}$. However, even if $a$ is not realised as the Euler class of any taut foliation on $M$, it is possible that $p^*(a)$ is realised by a taut foliation on $\tilde{M}$. The following question was asked by the author in \cite{yazdi2020thurston}. 
	
	\begin{question}
		Let $M$ be a closed hyperbolic 3-manifold with positive first Betti number. Let $a \in H^2(M ; \mathbb{R})$ be an integral point of norm equal to (respectively at most) one and satisfying the parity condition. Is there a finite covering $p \colon \tilde{M} \rightarrow M$ and a taut foliation on $\tilde{M}$ whose Euler class is equal to $p^*(a)$?
		\label{que:virtual-taut foliation}
	\end{question}

	Let $p \colon \tilde{M} \rightarrow M$ be a covering map. Then $p$ induces a map $p^* \colon H^2(M ; \mathbb{R}) \rightarrow H^2(\tilde{M}; \mathbb{R})$. There is also a map in the opposite direction called the \emph{Umkehr homomorphism}\index{Umkehr homomorphism} 
	\[ p! \colon H^2(\tilde{M}; \mathbb{R}) \rightarrow H^2(M ; \mathbb{R})\] 
	defined as $p! = \mathrm{PD} \circ p_* \circ \mathrm{PD}$, where $\mathrm{PD}$ is the Poincar\'{e} duality map, and $p_* \colon H_1(\tilde{M}; \mathbb{R}) \rightarrow H_1(M ; \mathbb{R})$ is the induced map on homology. The Umkehr homomorphism is also called the transfer homomorphism, but we will use Umkehr here since we have already used the word transfer map in a different meaning for the map on bounded cohomology. Note that if $a \in H^2(M ; \mathbb{R})$ then $p! (p^*(a)) = [\tilde{M}: M] \cdot a$. See \cite[Section 3G]{hatcher2002algebraic} for basic properties of Umkehr homomorphisms. 
	
	The following result of Yi Liu \cite[Lemma 5.1]{liu2024criterion} gives some evidence for a positive answer to the above question. It shows that given any \emph{rational} cohomology class $a \in H^2(M ; \mathbb{R})$ of dual norm one, there exists a finite cyclic cover $p \colon \tilde{M} \rightarrow M$ and a taut foliation $\tilde{\mathcal{F}}$ of $\tilde{M}$ such that the \emph{projection} of the Euler class of $\tilde{\mathcal{F}}$ to $H^2(M ; \mathbb{R})$ under the Umkehr homomorphism $p! \colon H^2(\tilde{M}; \mathbb{R}) \rightarrow H^2(M ; \mathbb{R})$ is the same as the \emph{projection} of $p^*(a)$ to $H^2(M ; \mathbb{R})$. 
	
	\begin{thm}[Liu]
		Let $M$ be a closed oriented hyperbolic 3-manifold. Denote by $B_x(M)$ the unit ball of the Thurston norm of $M$, and by $B_{x^*}(M)$ the unit ball of the dual norm. Let $F \subset \partial B_{x^*}(M)$ and $F^\wedge \subset \partial B_x(M)$ be a dual pair of closed faces.  Suppose that $w \in H^2(M ; \mathbb{R})$ is a rational point in the interior of $F$, and $\Sigma \in H_2(M ; \mathbb{Z})$ is a primitive homology class such that $\Sigma$ lies in the interior of $F^\wedge$. 
		There exists some finite cyclic covering $p \colon \tilde{M} \rightarrow M$ dual to $\Sigma$, and some transversely oriented taut foliation $\tilde{\mathcal{F}}$ on $\tilde{M}$ such that the following equality holds in $H^2(M ; \mathbb{R})$
		\[ p!(e(\mathcal{F})) = [\tilde{M}: M] \cdot w. \]
		\label{thm:Liu-virtual-taut foliation-projection}
	\end{thm}

We now give a sketch of Liu's proof of Theorem \ref{thm:Liu-virtual-taut foliation-projection} following \cite{liu2024criterion}. The proof uses the fully marked surface theorem (Theorem \ref{thm:fully marked surface}), the Euler class one theorem for vertices (Theorem \ref{thm:Gabai-euler class one for vertices}), and a construction, due to Liu, of assembling taut foliations together called \emph{medley construction}\index{medley construction}.

\begin{proof}[Sketch of the proof of Theorem \ref{thm:Liu-virtual-taut foliation-projection}]
	
	Let $v_1, \cdots, v_n \in F$ be the set of vertices of $F$. By the Euler class one theorem for vertices (Theorem \ref{thm:Gabai-euler class one for vertices}), for each $v_i$, there exists a taut foliation $\mathcal{F}_i$ on $M$ with Euler class $e(\mathcal{F}_i) = v_i$. Since $\Sigma$ is in the cone over $F^\wedge$, we have the following equality for every $i$
	\begin{eqnarray}
		\langle e(\mathcal{F}_i) , \Sigma \rangle = x(\Sigma), 
		\label{eq:fully marked}
	\end{eqnarray}
	where $x$ denotes the Thurston norm.
	By the fully marked surface theorem (Theorem \ref{thm:fully marked surface}), possibly after replacing $\mathcal{F}_i$ with new foliations but preserving Equality (\ref{eq:fully marked}), we may assume that for each $i$ there exists a closed (possibly disconnected) embedded oriented surface $S_i \subset M$ such that $[S_i] = \Sigma$, and $S_i$ is a union of compact leaves of $\mathcal{F}_i$. In general $S_i$ would be in different isotopy classes and possibly intersect each other, but to simplify the exposition and to bring the main ideas across here we make the assumption that $\Sigma$ is the unique norm-minimising surface in its homology class. For the complete proof without this assumption see Liu \cite{liu2024criterion}. Because of our simplifying assumption, $S_i = \Sigma$ up to isotopy, and so $\Sigma$ is a common leaf of $\mathcal{F}_1, \cdots, \mathcal{F}_n$. 
	
	Since $w \in \mathrm{int} (F) $ we can write it as 
	\[ w =  \frac{a_1 v_1 + \cdots + a_n v_n}{a_1 + \cdots + a_n},\]
	for positive integers $a_1 , \cdots, a_n$. Consider the foliation $\mathcal{F}_i \setminus \setminus \Sigma$ obtained by cutting $\mathcal{F}_i$ along the compact leaf $\Sigma$. There are two copies of $\Sigma$ in the boundary of $\mathcal{F}_i \setminus \setminus \Sigma$, denote them by $\Sigma_+$ and $\Sigma_-$. Consider a cyclic cover $\tilde{M}$ of $M$ dual to $\Sigma$ and of degree $(a_1 + \cdots + a_n)$, obtained by stacking $a_i$ copies of $\mathcal{F}_i \setminus \setminus \Sigma$ together for all $1 \leq i \leq n$, glued along the copies of $\Sigma$ such that $\Sigma_+$ in one copy is glued to $\Sigma_-$ in the adjacent copy. By construction $\tilde{M}$ comes equipped with a foliation $\mathcal{F}$ obtained by gluing a suitable number of copies of the foliations $\mathcal{F}_i \setminus \setminus \Sigma$ together. It is shown in \cite[Lemma 4.3]{liu2024criterion} that the Euler class $e(\mathcal{F})$ satisfies 
	\begin{eqnarray}
		p!(e(\mathcal{F})) = a_1 v_1 + \cdots + a_n v_n,
		\label{eq:euler class of medley construction}
	\end{eqnarray}
	or equivalently
	\[  p!(e(\mathcal{F}))  = (a_1 + \cdots + a_n) w = [\tilde{M} : M] \cdot w. \]
	We now justify Equality (\ref{eq:euler class of medley construction}) by a topological argument. Both sides of Equality (\ref{eq:euler class of medley construction}) are elements of $H^2(M ; \mathbb{R})$, and so it is enough to show that for every embedded oriented surface $S$ in $M$, the evaluations of both sides on the homology class $[S]$ are equal to each other. By definition of the transfer map we have 
	\[ \langle p!(e(\mathcal{F}))  , [S] \rangle = \langle e(\mathcal{F}) , [p^{-1}(S)] \rangle . \]
	So it is enough to show that 
	\begin{align*} \langle e(\mathcal{F}) , [p^{-1}(S)] \rangle  &= \langle  a_1 v_1 + \cdots + a_n v_n , [S] \rangle =  a_1 \langle e(\mathcal{F}_1) , [S] \rangle + \cdots + a_n  \langle e(\mathcal{F}_n) , [S] \rangle.
	\end{align*}
	Isotope $S$ so that it is in general position with respect to $\Sigma$ and intersects $\Sigma$ in a union of simple closed curves. Then $\langle e(\mathcal{F}_i) , [S] \rangle $ is the obstruction for finding a section of $T \mathcal{F}_i$ over $S$. Define a common section of $T \mathcal{F}_i$ over $\Sigma \cap S$ as any non-vanishing vector field $s$ in $T \Sigma \cap T S$. Let $M'$ be the manifold obtained by cutting $M$ along $\Sigma$, and $\mathcal{F}_i'$ be the foliation on $M'$ obtained by cutting $\mathcal{F}_i$ along the compact leaf $\Sigma$.  Therefore $\langle e(\mathcal{F}_i) , [S] \rangle $ is equal to the obstruction for extending the section $s$ to a section of $T \mathcal{F}_i'$ over $S'$. Now the foliation $\mathcal{F}$ is obtained by stacking together $a_i$ copies of $\mathcal{F}_i'$, glued along copies of $\Sigma$. The section $s$ lifts to a section $\tilde{s}$ of $\mathcal{F}$ on $p^{-1}(\Sigma \cap S) = p^{-1}(\Sigma) \cap p^{-1}(S) $, and $\langle e(\mathcal{F}) , [p^{-1}(S)] \rangle $ is the obstruction for extending $\tilde{s}$ to a section over $p^{-1}(S)$. Now the latter obstruction is the sum of obstructions coming from extending the section $\tilde{s}$ to the part of $p^{-1}(S)$ lying inside $a_i$ copies of $\mathcal{F}_i'$ for $1 \leq i \leq n$, which by the previous discussion is equal to $ a_i \langle e(\mathcal{F}_i) , [S] \rangle $. The equality follows. 

\end{proof}
	
	Using this, Liu \cite{liu2024criterion} gave a criterion for Question \ref{que:virtual-taut foliation} to have a positive answer in terms of nonvanishing of Alexander polynomials. His criterion, when satisfied, allows to realise any \emph{rational} point on certain closed faces on the boundary of the dual Thurston norm unit ball. In view of Liu's criterion, the parity condition seems less relevant for Question \ref{que:virtual-taut foliation}, at least when the dual norm of $a$ is equal to one. The following is the combination of \cite[Theorem 1.2]{liu2024criterion} and \cite[Corollary 1.3]{liu2024criterion}.
	
	\theoremstyle{theorem}
	\newtheorem*{Liu}{Theorem \ref{thm: Liu-virtual-taut foliation}}
	\begin{Liu}
		Let $M$ be a closed oriented hyperbolic 3-manifold. Denote by $B_x$ the unit ball of the Thurston norm of $M$, and by $B_{x^*}$ the unit ball of the dual norm. Let $F \subset \partial B_{x^*}$ and $F^\wedge \subset \partial B_{x}$ be a dual pair of closed faces. Suppose that $\psi \in H^1(M ; \mathbb{Z}) $ is a primitive cohomology class such that the Poincar\'{e} dual of $\psi$ lies in the cone over the interior of $F^\wedge$. 
		If the Alexander polynomial $\Delta_M^{\psi}(t)$ does not vanish, then for any rational point $w$ in $F$, there exists some finite cyclic cover $\tilde{M}$ of $M$ dual to $\psi$, such that the pullback of $w$ to $\tilde{M}$ is the real Euler class of some taut foliation on $\tilde{M}$.
	\end{Liu}	
	
	As a corollary, Liu gives examples of hyperbolic 3-manifolds with first Betti numbers 2 and 3 respectively such that every \emph{rational} point on the boundary of the unit ball of the dual Thurston norm is virtually realised as the real Euler class of a taut foliation. We now give a sketch of Liu's proof of Theorem \ref{thm: Liu-virtual-taut foliation} following \cite{liu2024criterion}.
	
	\begin{proof}[Sketch of proof of Theorem \ref{thm: Liu-virtual-taut foliation}] Let $M$ be a closed hyperbolic 3-manifold and $\psi \in H^1(M ; \mathbb{Z})$ be a primitive cohomology class such that the Alexander polynomial $\Delta_M^\psi(t)$ does not vanish. The nonvanishing of the Alexander polynomial implies that for every finite \emph{cyclic} cover of $M$ dual to $\psi$, the first Betti number of the cover is at most 
	\[ \deg(\Delta_M^\psi(t))+1. \]
	Here $\deg(\Delta_M^\psi(t))$ is defined as the difference between the degrees of the highest and the lowest power of $t$. The important point here is that the bound does not depend on the degree of the cyclic cover. See \cite[Lemma 3.3]{liu2024criterion} for a proof of this fact. Let $M'$ be a finite cyclic cover of $M$ dual to $\psi$ such that $b_1(M')$ is maximal. 
	
	Setting $d = [M' : M]$, the pullback $\psi' \in H^1(M' ; \mathbb{Z})$ has divisibility $d$, and so $\Sigma' = \mathrm{PD}(\psi'/d)$ is a primitive cohomology class. For simplicity we assume that $\omega$ lies in the \emph{interior} of $F$, the case that $\omega$ lies in the boundary of $F$ follows from this, see \cite[page 2]{liu2024criterion} for this deduction. Let $w' \in H^2(M' ;\mathbb{R})$ be the pullback of $w$ to $M'$. If $F' \subset \partial B_{x^*}(M')$ is the minimal closed face containing the pullback of $F$, then $w'$ lies in the interior of $F'$. Applying Theorem \ref{thm:Liu-virtual-taut foliation-projection} to $M'$ with respect to $w'$ and $\Sigma'$, we obtain some finite cyclic covering $p \colon \tilde{M} \rightarrow M'$ dual to $\Sigma'$, and some taut foliation $\tilde{\mathcal{F}}$ on $\tilde{M}$ such that the projection of the Euler class $e(\tilde{\mathcal{F}})$ to $H^2(M' ; \mathbb{R})$ under the Umkehr homomorphism $p! \colon H^2(\tilde{M}; \mathbb{R}) \rightarrow H^2(M' ; \mathbb{R})$ is equal to $[\tilde{M} : M'] \cdot w'$. 
	
	Since we assumed that $M'$ already maximises the first Betti number between all cyclic covers of $M$ dual to $\psi$, we have $b_1(\tilde{M}) = b_1(M')$. Therefore the Umkehr homomorphism $p! \colon H^2(\tilde{M} ; \mathbb{R}) \rightarrow H^2(M' ; \mathbb{R})$ is an isomorphism, since $p! = \mathrm{PD} \circ p_* \circ \mathrm{PD}$ is a composition of isomorphisms. Hence $e(\tilde{\mathcal{F}})$ is equal to the pullback of $w' \in H^2(M' ; \mathbb{R})$ to $\tilde{M}$, as their images under the injective map $p!$ are the same. But the pullback of $w'$ to $\tilde{M}$ is equal to the pullback of $w$ to $\tilde{M}$, and so $e(\tilde{\mathcal{F}})$ is also equal to the pullback of $w$ to $\tilde{M}$, completing the proof.

	\end{proof}

	\bibliographystyle{alpha}
	\bibliography{Reference-foliation-and-contact}
	
	\printindex
\end{document}